\documentclass[a4paper,11pt, reqno]{amsart}
\usepackage{amssymb,amsthm,amsmath}
\usepackage{cite}

\usepackage{amsmath,amsthm,amscd,amssymb}
\usepackage{latexsym}
\usepackage[colorlinks,citecolor=red,pagebackref,hypertexnames=false]{hyperref}
 
\numberwithin{equation}{section}

\theoremstyle{plain}
\newtheorem{theorem}{Theorem}[section]

\newtheorem{corollary}[theorem]{Corollary}
\newtheorem{proposition}[theorem]{Proposition}

\theoremstyle{definition}
\newtheorem{definition}[theorem]{Definition}

\theoremstyle{remark}
\newtheorem{remark}[theorem]{Remark}

\newtheorem{case[theorem]}{Case}

\def \R{{\mathbb R}}
\def \N{{\mathbb N}}

\def \Z{{\mathbb Z}}

\def \T{{\mathbb T}}

\def\supp{\hbox{supp\,}}
\def\norm#1.#2.{\lVert#1\rVert_{#2}}

\def\R{\mathbb R}

\title[Uncertainty principles for the short-time Fourier transform on the lattice]{Uncertainty principles for the short-time Fourier transform on the lattice}

\author{Anirudha Poria}
\author{Aparajita Dasgupta} 

\thanks{Research supported by Core Research Grant (RP03890G), Science and Engineering Research Board (SERB), DST, India.}

\address{Department of Applied Mathematics, School of Mathematics and Physics, Xi’an Jiaotong-Liverpool University,  Suzhou 215123, Jiangsu, China}
\email{Anirudha.Poria@xjtlu.edu.cn, anirudhamath@gmail.com}

\address{Department of Mathematics, Indian Institute of Technology Delhi, New Delhi 110016, India} 
\email{adasgupta@maths.iitd.ac.in}



\date{\today}
\begin{document}
\maketitle
\begin{abstract} 
In this paper, we study a few  versions of  the uncertainty principle for the short-time Fourier transform on the lattice $\Z^n \times \T^n$. In particular, we establish the  uncertainty principle for orthonormal sequences, Donoho--Stark's uncertainty principle,  Benedicks-type uncertainty principle, Heisenberg-type uncertainty principle and local uncertainty inequality for this transform on $\Z^n \times \T^n$. Also, we obtain the Heisenberg-type uncertainty inequality using the $k$-entropy of the short-time Fourier transform on $\Z^n \times \T^n$.
\end{abstract}

\section{Introduction}

In classical harmonic analysis, uncertainty principles are inequalities or uniqueness theorems concerning the joint localization of a function and its spectrum. Uncertainty principles play a fundamental role in the field of mathematics, physics, in addition to some engineering areas such as quantum theory, image processing, signal processing, optics, and many other well-known areas \cite{ric14, gro01, dono89, bial85}.  In quantum physics, it says that a particle’s speed and position cannot both be measured with infinite precision. The classical Heisenberg uncertainty principle was established in the Schr\"odinger space (the square-integrable function space). It states that a non-zero function and its Fourier transform cannot be both sharply localized, i.e.,  it is impossible for a non-zero function and its Fourier transform to be simultaneously small. This phenomenon has been under intensive study for almost a century now and extensively investigated in different settings. One can formulate different forms of the uncertainty principle depending on various ways of measuring the localization of a function. Uncertainty principles can be divided into two  categories: quantitative and qualitative uncertainty principles. Quantitative uncertainty principles are some special types of inequalities that tells us how a function and its Fourier transform relate. For example,  Benedicks \cite{ben95}, Donoho and Stark \cite{dono89}, and  Slepian and Pollak \cite{sle61} gave quantitative uncertainty principles for the Fourier transforms.  On the other hand, qualitative uncertainty principles   imply the vanishing of a function under some strong conditions on the function. In particular, Hardy \cite{har33}, Morgan \cite{mor34}, Cowling and Price \cite{cow83}, and Beurling \cite{hor1991} theorems are the  examples of qualitative uncertainty principles. For a more detailed study on the history of the uncertainty principle, and for many other generalizations and variations of the uncertainty principle, we refer to the book of Havin and J\"oricke \cite{havin}, and the excellent survey of Folland and Sitaram \cite{fol97}.

In signal analysis and in quantum mechanics, uncertainty principles are often discussed for simultaneous time-frequency representations, such as the short-time Fourier transform (STFT) or the Wigner distribution. These time-frequency representations are transforms which map a function on $\R^n$ to a function on $\R^n \times \widehat{\R}^n$, called the time-frequency plane or phase space. Since the dual group of $\mathbb{R}^{n}$ is the same as $\mathbb{R}^{n}$, the phase space is $\mathbb{R}^{n} \times \mathbb{R}^{n}$.  For the group $\mathbb{Z}^n$, the dual group is $\T^n$ and the phase space is $\mathbb{Z}^n \times \T^n$. In this paper, we consider the STFT on $\mathbb{Z}^n \times \T^n$ and study several uncertainty principles. The main results we discuss here for the STFT on $\mathbb{Z}^n \times \T^n$ are analogues of uncertainty principles for the STFT on $\mathbb{R}^{n} \times \mathbb{R}^{n}$. Gr\"ochenig and Zimmermann \cite{Gro01} showed that it is possible to derive new uncertainty principles for the STFT from uncertainty principles for the pair $(f, \widehat{f})$ using a fundamental identity for the STFT. Also, Gr\"ochenig in \cite{Gro03}, studied several uncertainty principles for time-frequency representations. Over the years, discovering new mathematical formulations of the uncertainty principle for the  STFT have drawn significant attention among many researchers, see for example \cite{bona03, Wilc00, gro01, havin} and the references therein. However, upto our knowledge, quantitative uncertainty principles have not been studied for the STFT on the phase space $\Z^n \times \T^n$. In this paper, we attempt to study several quantitative uncertainty principles for the STFT on $\Z^n \times \T^n$. 

It is natural to ask whether uncertainty principles involving $f$ and $\widehat{f}$ possess analogous statements in terms of the STFT. A few methods for transferring localization statements for Fourier pairs to corresponding statements about localization of time-frequency distributions were given in \cite{Gro03}. In some cases, the extra rotational symmetry of the plane provides more refined information than does separate consideration of $f$ and $\hat{f}$. The motivation to study quantitative uncertainty principles for the STFT on the phase space $\Z^n \times \T^n$ arises from the classical uncertainty principles for the STFT and the remarkable contribution of this transform in time-frequency analysis (see \cite{gro01}). Time-frequency localization operators are a mathematical tool to define a restriction of functions to a region in the time-frequency plane to extract time-frequency features. The role of these operators is to localize a signal simultaneously in time and frequency domains, this can be seen as the uncertainty principle. Recent works in localization operators on $\Z^n$ (see \cite{das22}) motivated us to study uncertainty principles for the STFT on the phase space $\Z^n \times \T^n$. The study of operators defined by phase space functions is sometimes called microlocal analysis. We hope that the study of uncertainty principles for the STFT on the phase space $\Z^n \times \T^n$ makes a significant impact in microlocal analysis.

The main aim of this paper is to study a few  uncertainty principles related to the STFT on $\Z^n \times \T^n$. More preciously, we prove the uncertainty principle for orthonormal sequences, Donoho--Stark's uncertainty principle, Benedicks-type uncertainty principle, Heisenberg-type uncertainty principle and local uncertainty inequality for the STFT on $\Z^n \times \T^n$. We study the version of Donoho--Stark's uncertainty principle and show that the STFT cannot be concentrated in any small set. Also, we obtain an estimate for the size of the essential support of this transform under the condition that the STFT of a non-zero function is time-frequency concentrated on a measurable set. Then, we investigate the Benedicks-type uncertainty principle and show that the STFT cannot be concentrated inside a set of measures arbitrarily small. Further, we study the Heisenberg-type uncertainty inequality for a general magnitude and provide the result related to the $L^2(\mathbb{Z}^n \times \T^n )$-mass of the STFT outside sets of finite measure. Finally, we obtain the Heisenberg-type uncertainty inequality using the $k$-entropy and study the localization of the $k$-entropy of the STFT on $\Z^n \times \T^n$. 

The paper is organized as follows. In Section \ref{sec2}, we present some essential properties related to the STFT on $\Z^n \times \T^n$. Some different types of uncertainty principles associated with the STFT are provided in Section \ref{sec4}. In particular,  we prove the uncertainty principle for orthonormal sequences, Donoho--Stark's uncertainty principle, Benedicks-type uncertainty principle, Heisenberg-type uncertainty principle and local uncertainty inequality for the STFT on $\Z^n \times \T^n$. We conclude the paper with the Heisenberg-type uncertainty inequality  using the $k$-entropy of this transform.

\section{Short-time Fourier transform on $\Z^n \times \T^n$}\label{sec2}

The Schwartz space $\mathcal{S}\left(\mathbb{Z}^{n}\right)$ on $\mathbb{Z}^{n}$ is the space of rapidly decreasing functions $g: \mathbb{Z}^{n} \rightarrow \mathbb{C}$, i.e. $g \in \mathcal{S}\left(\mathbb{Z}^{n}\right)$ if for any $M<\infty$ there exists a constant $C_{g, M}$ such that 
$$|g(k)| \leq C_{g, M} (1+|k|)^{-M}, \quad \text { for all } k \in \mathbb{Z}^{n}.$$ 
The topology on $\mathcal{S}\left(\mathbb{Z}^{n}\right)$ is given by the seminorms $p_{j}$, where $j \in \mathbb{N}_{0}= \N \cup \{0\}$ and $p_{j}(g):=\sup\limits_{k \in \mathbb{Z}^{n}}(1+|k|)^{j}|g(k)|$. The space $\mathcal{S}^{\prime}\left(\mathbb{Z}^{n}\right)$ of all continuous linear functionals on $\mathcal{S}\left(\mathbb{Z}^{n}\right)$ is called the space of tempered distributions. 

Fix $k \in \Z^n$, $w \in \T^n$ and $f \in \ell^2(\Z^n)$. For $m \in \Z^n$, the translation operator $T_k$ is defined by $T_k f(m)=f(m-k)$ and the modulation operator $M_w$ is defined by $M_w f(m)=e^{2 \pi i w \cdot m} f(m)$. Let $g \in \mathcal{S}\left(\mathbb{Z}^{n}\right)$ be a fixed window function. Then, the STFT of a function $f \in \mathcal{S}^{\prime}\left(\mathbb{Z}^{n}\right)$ with respect to $g$ is defined to be the function on $\Z^n \times \T^n$ given by
\[ V_g f (m, w)= \left\langle f, M_w T_m g \right\rangle_{\ell^2(\Z^n)} =\sum_{k \in \Z^n} f(k) \overline{M_w T_m g (k)} =\sum_{k \in \Z^n} f(k) \overline{g(k-m)} e^{-2 \pi i w \cdot k}.\]
For $k \in \Z^n$, we define $\tilde{g}(k)=g(-k)$. Then, we can write $V_g f$ as a convolution on $\Z^n$
\[ V_g f (m, w)= e^{-2 \pi i w \cdot m} \left( f * M_w \overline{\tilde{g}} \right)(m). \]

For $1 \leq p<\infty$, we define $L^{p}\left(\mathbb{Z}^n \times \T^n \right)$ to be the space of all measurable functions $H$ on $\mathbb{Z}^n \times \T^n$ such that
\[ \|H\|_{L^{p}\left(\mathbb{Z}^n \times \T^n \right)}^{p}= \sum_{k \in \mathbb{Z}^n} \int_{\T^n}  |H(k, w)|^{p} \;dw < \infty .\]

Next, we present some basic properties of the STFT on $\Z^n \times \T^n$. These are analogues properties of the STFT for functions on $\mathbb{R}^n$ (see \cite{gro01}).
\begin{proposition}
\begin{itemize}
\item[(1)] $($Orthogonality relation$)$  For every $f_1, f_2, g_1, g_2 \in \ell^2(\Z^n)$, we have
\begin{equation}\label{eq18}
\left\langle V_{g_1} f_1, V_{g_2} f_2 \right\rangle_{L^2(\Z^n \times \T^n)}=\langle f_1, f_2 \rangle_{\ell^2(\Z^n)} \; \langle g_2, g_1 \rangle_{\ell^2(\Z^n)},
\end{equation}
where $\langle\cdot, \cdot\rangle_{L^2(\Z^n \times \T^n)}$ denotes the inner product of $L^2(\Z^n \times \T^n)$.

\item[(2)] $($Plancherel's formula$)$ Let $g \in \ell^2(\Z^n)$ be a non-zero window function. Then for every $f \in \ell^2(\Z^n)$, we have
\begin{equation}\label{eq17}
\left\Vert V_g f \right\Vert_{L^2(\Z^n \times \T^n)} = \Vert f \Vert_{\ell^2(\Z^n)} \; \Vert g \Vert_{\ell^2(\Z^n)}.
\end{equation}

\item[(3)] $($Inversion formula$)$ Let $g, \gamma \in \ell^2(\Z^n)$ and $\langle g, \gamma\rangle_{\ell^2(\Z^n)} \neq 0$. Then for every $f \in \ell^2(\Z^n)$, we have
\[f=\frac{1}{\langle\gamma, g\rangle_{\ell^2(\Z^n)}}
\sum_{m \in \Z^n} \int_{\T^n} V_g f(m,w) \; M_w T_m \gamma  \; dw.\]
\end{itemize}
\end{proposition}
\begin{proof}
(1) Note that we can write $V_g f(m,w)=\mathcal{F}_{\Z^n} (f \cdot T_m\overline{g})(w)$. Using Parseval's formula, we obtain
\begin{eqnarray*}
\left\langle V_{g_1} f_1, V_{g_2} f_2 \right\rangle_{L^2(\Z^n \times \T^n)}
&=& \sum_{m \in \Z^n} \int_{\T^n} V_{g_1} f_1(m,w) \; \overline{V_{g_2} f_2(m,w)} \; dw    \\
&=& \sum_{m \in \Z^n} \int_{\T^n} \mathcal{F}_{\Z^n} (f_1 \cdot T_m\overline{g}_1)(w) \; \overline{\mathcal{F}_{\Z^n} (f_2 \cdot T_m\overline{g}_2)(w)} \; dw  \\
&=& \sum_{m \in \Z^n} \sum_{k \in \Z^n}  (f_1 \cdot T_m\overline{g}_1)(k) \; \overline{(f_2 \cdot T_m\overline{g}_2)(k)} \\
&=& \sum_{m \in \Z^n} \sum_{k \in \Z^n} f_1(k) \;  \overline{g_1(k-m)} \; \overline{f_2(k)}\; g_2(k-m) \\
&=& \sum_{k \in \Z^n} f_1(k) \; \overline{f_2(k)} \sum_{m \in \Z^n} g_2(m)\; \overline{g_1(m)} \\
&=& \langle f_1, f_2 \rangle_{\ell^2(\Z^n)} \; \langle g_2, g_1 \rangle_{\ell^2(\Z^n)}.
\end{eqnarray*}

(2) By taking $g_1=g_2=g$ and $f_1=f_2=f$ in the relation (\ref{eq18}), we obtain the Plancherel formula (\ref{eq17}).

(3) Since $V_g f \in L^2\left(\Z^n \times \T^n  \right)$, the integral 
\[\tilde{f}=\frac{1}{\langle\gamma, g\rangle_{\ell^2(\Z^n)}} \sum_{m \in \Z^n} \int_{\T^n} V_g f(m,w) \; M_w T_m \gamma  \; dw\]
is a well-defined function in $\ell^2(\Z^n)$. Also, for any $h \in \ell^2(\Z^n)$, using the orthogonality relations, we obtain that
\begin{eqnarray*}
\langle \tilde{f}, h \rangle_{\ell^2(\Z^n)}
& = & \frac{1}{\langle\gamma, g\rangle_{\ell^2(\Z^n)}} \sum_{m \in \Z^n} \int_{\T^n} V_g f(m,w) \; \overline{\left\langle h, M_w T_m \gamma \right\rangle}_{\ell^2(\Z^n)} \;dw \\
& = & \frac{1}{\langle\gamma, g\rangle_{\ell^2(\Z^n)}} \left\langle V_g f, V_{\gamma} h \right\rangle_{L^2\left(\Z^n \times \T^n \right)} =\langle f, h\rangle_{\ell^2(\Z^n)}.   
\end{eqnarray*}
Thus $\tilde{f}=f$, and we obtain the inversion formula.
\end{proof} 

Let $g \in \ell^2(\Z^n)$ be a non-zero window function. Then for every $f \in \ell^2(\Z^n)$, using the Cauchy--Schwarz inequality, we get  
\[ \left| V_g f (m, w) \right| 
= \left| \left\langle f, M_w T_m g \right\rangle_{\ell^2(\Z^n)} \right| 
\leq \Vert f \Vert_{\ell^2(\Z^n)} \left\Vert M_w T_m g \right\Vert_{\ell^2(\Z^n)}
\leq \Vert f \Vert_{\ell^2(\Z^n)} \left\Vert g \right\Vert_{\ell^2(\Z^n)}. \]
Therefore,
\begin{equation}\label{eq21}
\left\Vert V_{g} f \right\Vert_{L^\infty(\Z^n \times \T^n)} \leq \Vert f \Vert_{\ell^2(\Z^n)} \; \Vert g \Vert_{\ell^2(\Z^n)}.
\end{equation}
Also, using Plancherel's formula (\ref{eq17}), relation (\ref{eq21}) and the Riesz--Thorin interpolation theorem (see \cite{ste56}), we obtain that the function $V_{g} f \in L^p(\Z^n \times \T^n)$, for $p \in [2, \infty)$ and 
\begin{equation}\label{eq22}
\left\Vert V_{g} f \right\Vert_{L^p(\Z^n \times \T^n)} \leq \Vert f \Vert_{\ell^2(\Z^n)} \; \Vert g \Vert_{\ell^2(\Z^n)}.
\end{equation}  
We define the measure $\nu \otimes \mu$ on $\Z^n \times \T^n$ by 
\begin{equation}\label{eq16}
d(\nu \otimes \mu)(m, w)= d\nu(m)  \; d\mu(w),
\end{equation}
where $d\nu(m)$ is the counting measure and $d\mu(w)$ is the Lebesgue measure $dw$.

Next, we present a result on reproducing kernel Hilbert space for the STFT, which is the analogue properties of the STFT for functions on $\mathbb{R}^n$ (see \cite{gro01}). 
\begin{theorem}[Reproducing kernel Hilbert space]
The space $V_g \left(\ell^2(\Z^n) \right)$ is a reproducing kernel Hilbert space in $L^2(\Z^n \times \T^n)$ with kernel function $K_g$ defined by
\begin{equation}\label{eq19}
K_{g}((m',w') ; (m,w)) = \frac{1}{\|g\|_{\ell^2(\Z^n)}^2} V_g \left( M_w T_m g \right) (m',w').
\end{equation}
Moreover, the kernel is pointwise bounded
\[ \left|K_{g}((m',w') ; (m,w)) \right| \leq 1, \quad \text{for all }  \; (m,w), \; (m',w') \in \Z^n \times \T^n.   \]
\end{theorem} 
\begin{proof}
Using the relation \eqref{eq18}, we obtain that
\begin{eqnarray*}
V_g f (m, w) 
& = & \left\langle f, M_w T_m g \right\rangle_{\ell^2(\Z^n)} \\
& = & \frac{1}{\|g\|_{\ell^2(\Z^n)}^2} \left\langle V_g f,V_g ( M_w T_m g) \right\rangle_{L^2(\Z^n \times \T^n)} \\
& = & \frac{1}{\|g\|_{\ell^2(\Z^n)}^2}  \sum_{m' \in \Z^n}  \int_{\T^n} V_g f(m', w') \;  \overline{V_g(M_w T_m g)(m' , w' )} \; dw' \\
&=& \frac{1}{\|g\|_{\ell^2(\Z^n)}^2}  \sum_{m' \in \Z^n}  \int_{\T^n} V_g f(m', w') \;  \overline{\left\langle M_w T_m g, M_{w'} T_{m'} g \right\rangle_{\ell^2(\Z^n)} } \; dw' \\
&=& \left\langle V_g f, K_{g}((\cdot, \cdot) ; (m,w)) \right\rangle_{L^2(\Z^n \times \T^n)},
\end{eqnarray*}
where $K_g$ is the function on $\Z^n \times \T^n$ given by 
\[ K_{g}((m',w') ; (m,w)) = \frac{1}{\|g\|_{\ell^2(\Z^n)}^2} V_g \left( M_w T_m g \right) (m',w'). \]
Using the Cauchy--Schwarz inequality and relation (\ref{eq18}), we get
\begin{eqnarray*}
|V_g f (m, w) | 
& = & \left| \left\langle V_g f, K_{g}((\cdot, \cdot) ; (m,w)) \right\rangle_{L^2(\Z^n \times \T^n)} \right| \\
& = & \frac{1}{\|g\|_{\ell^2(\Z^n)}^2} \left| \left\langle V_g f,V_g ( M_w T_m g) \right\rangle_{L^2(\Z^n \times \T^n)} \right| \\
&\leq & \frac{1}{\|g\|_{\ell^2(\Z^n)}^2} \left| \langle f, M_w T_m g \rangle_{\ell^2(\Z^n)} \right| \|g\|_{\ell^2(\Z^n)}^2\\
&\leq & \| f \|_{\ell^2(\Z^n)} \| M_w T_m g \|_{\ell^2(\Z^n)} \\
&\leq & \| f \|_{\ell^2(\Z^n)} \| g \|_{\ell^2(\Z^n)}.
\end{eqnarray*}
Now, using the Cauchy--Schwarz inequality,  for every
$(m,w), \; (m',w') \in \Z^n \times \T^n$ we get
\begin{eqnarray*}
\left|K_{g}((m',w') ; (m,w)) \right| 
& = & \frac{1}{\|g\|_{\ell^2(\Z^n)}^2}  \left| V_g \left( M_w T_m g \right) (m',w') \right| \\
& \leq &  \frac{1}{\|g\|_{\ell^2(\Z^n)}^2}  \left\Vert M_w T_m g \right\Vert_{\ell^2(\Z^n)} \left\Vert M_{w'} T_{m'} g \right\Vert_{\ell^2(\Z^n)}
=1.
\end{eqnarray*}
Also, for every $(m,w)  \in \Z^n \times \T^n $, by using Plancherel's formula \eqref{eq17}, we obtain 
\[ \Vert K_{g}((\cdot, \cdot) ; (m,w)) \Vert_{L^2(\Z^n \times \T^n)} \leq 1.  \]
This proves that the kernel $K_g \in L^2(\Z^n \times \T^n)$ and is bounded.
\end{proof}  

For $1 \leq p<\infty$, the set of all measurable functions $F$ on $\mathbb{Z}^n$ such that
\[ \|F\|_{\ell^{p}(\mathbb{Z}^n)}^{p}=\sum_{k \in \mathbb{Z}^n}|F(k)|^{p}<\infty \]    
is denoted by $\ell^{p}(\mathbb{Z}^n)$. We define $L^{p}(\T^n)$ to be the set of all measurable functions $f$ on $\T^n$ for which
\[ \|f\|_{L^{p} (\T^n )}^{p}= \int_{\T^n} |f(w)|^{p} \;dw < \infty. \]
Next, we define the Fourier transform $\mathcal{F}_{\mathbb{Z}^n} F$ of $F \in \ell^{1}(\mathbb{Z}^n)$ to be the function on $\T^n$ by
\[ \left(\mathcal{F}_{\mathbb{Z}^n} F \right)(w)=\sum_{k \in \mathbb{Z}^n} e^{-2 \pi i k \cdot w} F(k), \quad w \in \T^n. \]
Let $f$ be a function on $\T^n$, then we define the Fourier transform $\mathcal{F}_{\T^n} f$ of $f$ to be the function on $\mathbb{Z}^n$ by
\[ \left(\mathcal{F}_{\T^n} f\right)(k)= \int_{\T^n} e^{2 \pi i k \cdot w} f(w) \;dw, \quad k \in \mathbb{Z}^n. \]
Note that $\mathcal{F}_{\mathbb{Z}^n}: \ell^{2}(\mathbb{Z}^n) \rightarrow L^{2} (\T^n)$ is a surjective isomorphism. 
Also, $\mathcal{F}_{\mathbb{Z}^n}=\mathcal{F}_{\T^n}^{-1}=\mathcal{F}_{\T^n}^{*} $ and $\left\|\mathcal{F}_{\mathbb{Z}^n} F\right\|_{L^{2}\left(\T^n\right)}=\|F\|_{\ell^{2}(\mathbb{Z}^n)},$ $F \in \ell^{2}(\mathbb{Z}^n)$.


\section{Uncertainty principles for the STFT on $\Z^n \times \T^n$}\label{sec4}
In this section, we study various uncertainty principles for the STFT on $\Z^n \times \T^n$. Mainly, we establish the uncertainty principle for orthonormal sequences, Donoho--Stark's uncertainty principle, Benedicks-type uncertainty principle, Heisenberg-type uncertainty principle and local uncertainty inequality for this transform on $\Z^n \times \T^n$. We begin with the uncertainty principle for orthonormal sequences.

\subsection{Uncertainty principle for orthonormal sequences}

Here, we obtain the uncertainty principle for orthonormal sequences associated with the STFT. We first define the following orthogonal projections: 

\begin{enumerate}
\item  Let $P_{g}$ be the orthogonal projection from $L^2(\Z^n \times \T^n)$ onto $V_g \left(\ell^{2}\left(\Z^n \right)\right) $ and   $\operatorname{Im}P_g$  denotes the range of $P_{g}$. Also, $P_g$ is called the orthogonal projection with respect to $g$.

\item Let $P_{\Sigma}$ be the orthogonal projection on $L^2(\Z^n \times \T^n)$ defined by
\begin{align}\label{eq25}
P_{\Sigma} F=\chi_{\Sigma} F, \quad F \in L^2(\Z^n \times \T^n),
\end{align}
where $\Sigma \subset \Z^n \times \T^n$, $\chi_{\Sigma}$ is the characteristic function on $\Sigma$, and $ \operatorname{Im}P_\Sigma $ is the range of $P_{\Sigma}$. The operator $P_{\Sigma}$ is called the orthogonal projection on $\Sigma$.
\end{enumerate}
Also, we define 
$$\left\|P_{\Sigma} P_{g}\right\|_{op}=\sup \left\{\left\|P_{\Sigma} P_{g}(F)\right\|_{L^2(\Z^n \times \T^n)}: F \in L^2(\Z^n \times \T^n), \|F\|_{L^2(\Z^n \times \T^n)}=1\right\}.$$ 
In the following, we obtain an estimate for $\left\|P_{\Sigma} P_{g}\right\|_{op}$.

\begin{theorem}\label{eq26}
Let $g \in \ell^2(\Z^n)$ be a non-zero window function. Then  for any  $\Sigma  \subset \Z^n \times \T^n$ of  finite measure $(\nu \otimes \mu)(\Sigma)<\infty$,  the operator $P_{\Sigma} P_{g}$ is a Hilbert--Schmidt  operator. Also, we have
$$ \left\|P_{\Sigma} P_{g}\right\|_{op}^{2} \leq (\nu \otimes \mu)(\Sigma).$$
\end{theorem}
\begin{proof}
Since $P_g$ is a projection onto a reproducing kernel Hilbert space, for any $F \in L^2(\Z^n \times \T^n)$, the orthogonal projection $P_{g}$ can be expressed as
$$ P_{g}(F)(m, w)=\sum_{m' \in \Z^n} \int_{\T^n} F(m', w') \; K_g ((m', w');(m, w))\; d\mu( w'), $$
where $K_g ((m', w');(m, w))$ is  given  by (\ref{eq19}). Using the relation (\ref{eq25}), we get
$$ P_{\Sigma} P_{g}(F)(m, w)= \sum_{m' \in \Z^n} \int_{\T^n} \chi_{\Sigma}(m, w) \; F(m', w') \; K_g ((m', w');(m, w))\; d\mu(w').$$
Hence, the operator $P_{\Sigma} P_{g}$  is an integral operator with kernel $$K((m', w');(m, w))=\chi_{\Sigma}(m, w) \;  K_g ((m', w');(m, w)). $$ 
Let $\Sigma= Z \times T $, where $Z$ and $T$ are some arbitrary sets in $\Z^n$ and $\T^n$ respectively. Using the relation \eqref{eq19} and Plancherel's formula (\ref{eq17}), we obtain 
\begin{align*}
\left\|P_{\Sigma} P_{g}\right\|_{H S}^{2}
&= \sum_{m \in \Z^n} \int_{\T^n} \left( \sum_{m' \in \Z^n} \int_{\T^n} \left|K((m', w');(m, w))\right|^2 \; d\mu (w') \right) d\mu(w) \\
&= \sum_{m \in \Z^n} \int_{\T^n} \left( \sum_{m' \in \Z^n} \int_{\T^n} \left|\chi_{\Sigma}(m, w)\right|^{2}\left| K_g ((m', w');(m, w))\right|^2  d\mu (w')\right) d\mu(w)\\
&= \frac{1}{\Vert g \Vert^4_{\ell^2(\Z^n)}} \sum_{m \in Z} \int_{T} \left(\sum_{m' \in \Z^n} \int_{\T^n}  \left| V_g \left( M_w T_m g  \right) (m', w') \right|^2 d\mu(w')\right) d\mu(w)\\
&\leq \frac{\Vert g \Vert^4_{\ell^2(\Z^n)}}{\Vert g \Vert^4_{\ell^2(\Z^n)}} \; (\nu \otimes \mu) (Z \times T)\\
& = (\nu \otimes \mu) (\Sigma). 
\end{align*}
Thus, the operator $P_{\Sigma} P_{g}$ is a   Hilbert--Schmidt operator. Now, the  proof  follows from the fact that $\left\|P_{\Sigma} P_{g}\right\|_{op} \leq \left\|P_{\Sigma} P_{g}\right\|_{H S}$.
\end{proof}

Next, we establish the uncertainty principle for orthonormal sequences associated with the STFT on $\Z^n \times \T^n$.

\begin{theorem}\label{thm3}
Let $\left\{ \phi_{n}\right\}_{n \in \mathbb{N}}$ be an orthonormal sequence in $\ell^2(\Z^n)$ and $g \in \ell^2(\Z^n)$ be a non-zero window function with $\Vert g \Vert_{\ell^2(\Z^n)}= 1$. Then for any $\Sigma \subset \Z^n \times \T^n$ of finite measure $(\nu \otimes \mu)(\Sigma)<\infty,$ we have
$$ \sum_{n=1}^{N}\left(1-\left\|\chi_{\Sigma^{c}} V_g \phi_{n}\right\|_{L^2(\Z^n \times \T^n)}\right) \leq (\nu \otimes \mu)(\Sigma), $$  
for every $N\in \mathbb{N}.$
\end{theorem}
\begin{proof}
Let $\left\{ e_{n}\right\}_{n \in \mathbb{N}}$ be an orthonormal basis for $L^{2}(\Z^n \times \T^n)$. Since $P_{\Sigma} P_{g}$ is a Hilbert--Schmidt operator, from Theorem \ref{eq26}, we obtain 
$$
tr \left(P_{g} P_{\Sigma} P_{g}\right)=\sum_{n \in \mathbb{N}}\left\langle P_{g} P_{\Sigma} P_{g} e_{n}, e_{n}\right\rangle_{L^2(\Z^n \times \T^n)}=\left\|P_{\Sigma} P_{g}\right\|_{H S}^{2} \leq (\nu \otimes \mu)(\Sigma),
$$
where $tr \left(P_{g} P_{\Sigma} P_{g}\right)$ denotes the trace of the operator $P_{g} P_{\Sigma} P_{g}$. Since $\left\{ \phi_{n}\right\}_{n \in \mathbb{N}}$ be an orthonormal sequence in $\ell^{2}(\Z^n)$,  from the  orthogonality relation (\ref{eq18}), we obtain that $\{ V_g \phi_{n}\}_{n \in \mathbb{N}}$ is also an orthonormal sequence in $L^2(\Z^n \times \T^n)$. Therefore
\begin{align*}
&\sum_{n=1}^{N}\left\langle P_{\Sigma} V_g \phi_{n}, V_g \phi_{n}\right\rangle_{L^2(\Z^n \times \T^n)} 
=\sum_{n=1}^{N}\left\langle P_{g} P_{\Sigma} P_{g} V_g \phi_{n}, V_g \phi_{n}\right\rangle_{L^2(\Z^n \times \T^n)} 
\leq tr\left(P_{g} P_{\Sigma} P_{g}\right).
\end{align*}
Thus, we have
$$ \sum_{n=1}^{N}\left\langle P_{\Sigma} V_g \phi_{n}, V_g \phi_{n}\right\rangle_{L^2(\Z^n \times \T^n)}  \leq (\nu \otimes \mu)(\Sigma).$$
Also, for any $n$ with $1 \leq n \leq N$, using the Cauchy--Schwarz inequality, we get 
$$
\begin{aligned}
&\left\langle P_{\Sigma} V_g \phi_{n}, V_g \phi_{n}\right\rangle_{L^2(\Z^n \times \T^n)} 
=1- \left\langle P_{\Sigma^c} V_g \phi_{n}, V_g \phi_{n}\right\rangle_{L^2(\Z^n \times \T^n)} 
\geq 1-\left\|\chi_{\Sigma^c} V_g \phi_{n}\right \|_{L^2(\Z^n \times \T^n)}.
\end{aligned}
$$
Therefore 
\begin{align*}
\sum_{n=1}^{N}\left(1-\left\|\chi_{\Sigma^c} V_g \phi_{n}\right \|_{L^2(\Z^n \times \T^n)}  \right) &\leq \sum_{n=1}^{N}\left\langle P_{\Sigma} V_g \phi_{n}, V_g \phi_{n}\right\rangle_{L^2(\Z^n \times \T^n)} 
\leq (\nu \otimes \mu)(\Sigma).
\end{align*}
This completes the proof of the theorem.
\end{proof}

\subsection{Donoho--Stark’s uncertainty principle for the STFT}
In this subsection, we investigate the version of Donoho--Stark’s uncertainty principle for the STFT on $\Z^n \times \T^n$. In particular, we study the case where $f$ and $V_g f$ are close to zero outside measurable sets. We start with the following result.
\begin{theorem}
Let $f \in \ell^2(\Z^n)$ such that $f \neq 0$ and $g \in \ell^2(\Z^n)$ be a non-zero window function. Then for any $\Sigma =Z \times T  \subset \Z^n \times \T^n$, where $Z$ and $T$ are some arbitrary sets in $\Z^n$ and $\T^n$ respectively, and $\varepsilon \geq 0$ such that
$$ \sum_{m \in Z} \int_{T} \left|V_g f (m, w)\right|^{2} \; d \mu(w) \geq (1-\varepsilon)\|f\|_{\ell^2(\Z^n)}^2 \|g\|_{\ell^2(\Z^n)}^2, $$
we have
$$ (\nu \otimes \mu) (\Sigma) \geq 1-\varepsilon. $$
\end{theorem}
\begin{proof}
Using the relation (\ref{eq21}), we get 
\begin{align*}
(1-\varepsilon)\|f\|_{\ell^2(\Z^n)}^2 \|g\|_{\ell^2(\Z^n)}^2
\leq \sum_{m \in Z} \int_{T} \left|V_g f(m, w)\right|^{2} \;d\mu(w)  
& \leq \left\|V_g f \right\|^2_{L^{\infty}\left(\Z^n \times \T^n\right)}\;  (\nu \otimes \mu) (Z \times T) \\
& \leq (\nu \otimes \mu)(\Sigma)\; \|f\|_{\ell^2(\Z^n)}^2 \|g\|_{\ell^2(\Z^n)}^2.
\end{align*}
Therefore, $(\nu \otimes \mu)(\Sigma) \geq 1-\varepsilon.$
\end{proof}

In the following proposition, we show that the STFT cannot be concentrated in any small set.

\begin{proposition}\label{eq35}
Let $g \in \ell^2(\Z^n)$ be a non-zero window function.
Then for any function $f \in \ell^2(\Z^n)$ and for any     $\Sigma \subset \Z^n \times \T^n $ such that $(\nu \otimes \mu)(\Sigma)<1$, we have 
$$\left\|\chi_{\Sigma^{c}} V_{g} f\right\|_{L^{2}\left(\Z^n \times \T^n\right)} \geq \sqrt{1-(\nu \otimes \mu)(\Sigma)}\;\|f\|_{\ell^2(\Z^n)} \|g\|_{\ell^2(\Z^n)} .$$
\end{proposition}
\begin{proof}
For any $f \in \ell^2(\Z^n)$, using the relation (\ref{eq21}), we get 
$$
\begin{aligned}
\left\|V_{g} f\right\|_{L^{2}\left(\Z^n \times \T^n\right)}^{2} 
& =\left\|\chi_{\Sigma} V_{g} f+\chi_{\Sigma^c} V_{g} f\right\|_{L^{2}\left(\Z^n \times \T^n\right)}^{2} \\
& = \left\|\chi_{\Sigma} V_{g} f\right\|_{L^{2}\left(\Z^n \times \T^n\right)}^{2}+\left\|\chi_{\Sigma^{c}} V_{g} f\right\|_{L^{2}\left(\Z^n \times \T^n\right)}^{2} \\
& \leq (\nu \otimes \mu)(\Sigma)\left\|V_{g} f\right\|_{L^{\infty}\left(\Z^n \times \T^n\right)}^{2}+\left\|\chi_{\Sigma^{c}} V_{g} f\right\|_{L^{2}\left(\Z^n \times \T^n\right)}^{2} \\
& \leq (\nu \otimes \mu)(\Sigma)\; \|f\|_{\ell^2(\Z^n)}^{2}\; \|g\|_{\ell^2(\Z^n)}^{2}+\left\|\chi_{\Sigma^{c}} V_{g} f\right\|_{L^{2}\left(\Z^n \times \T^n\right)}^{2}.
\end{aligned}
$$
Thus, using Plancherel's formula (\ref{eq17}), we obtain 
$$\left\|\chi_{\Sigma^{c}} V_{g} f\right\|_{L^{2}\left(\Z^n \times \T^n\right)} \geq \sqrt{1-(\nu \otimes \mu)(\Sigma)}\;\|f\|_{\ell^2(\Z^n)} \|g\|_{\ell^2(\Z^n)} .$$
\end{proof}

\begin{definition}
Let $E$ be a measurable subset of $\Z^n$ and $0 \leq \varepsilon_{E}<1$. Then we say that a  function  $f\in \ell^p(\Z^n), 1 \leq p \leq 2,$ is   $\varepsilon_{E}$-concentrated on $E$ in $\ell^p(\Z^n)$-norm, if
$$ \left\|\chi_{E^{c}} f\right\|_{\ell^p(\Z^n)} \leq \varepsilon_{E}\|f\|_{\ell^p(\Z^n)}.$$
If $\varepsilon_{E}=0$, then $E$ contains the support of  $f$ .
\end{definition}

\begin{definition}
Let $\Sigma$ be a measurable subset of $\Z^n \times \T^n$ and $0 \leq \varepsilon_{\Sigma}<1$.   Let $f, g \in \ell^2(\Z^n)$ be two non-zero functions. We say that $V_g f$ is $ \varepsilon_{\Sigma}$-time-frequency concentrated on $\Sigma$, if
$$ \left\|\chi_{\Sigma^{c}} V_{g} f\right\|_{L^{2}\left(\Z^n \times \T^n \right)} \leq \varepsilon_{\Sigma} \left\|  V_{g} f \right\|_{L^{2}\left(\Z^n \times \T^n \right)}.$$
\end{definition}

If $V_{g} f$ is $ \varepsilon_{\Sigma}$-time-frequency concentrated on $\Sigma$, then in the following, we obtain an estimate for the size of the  essential support of the STFT on $\Z^n \times \T^n$.
\begin{theorem}
Let $f \in \ell^2(\Z^n)$ such that $f \neq 0$ and $g \in \ell^2(\Z^n)$ be a non-zero window function. Let $\Sigma \subset \Z^n \times \T^n$ such that $(\nu \otimes \mu) (\Sigma)<\infty$ and $\varepsilon_\Sigma \geq 0$. If
$V_{g} f$ is $ \varepsilon_{\Sigma}$-time-frequency concentrated on $\Sigma,$ then 
$$(\nu \otimes \mu) (\Sigma) \geq (1-\varepsilon_\Sigma^2).$$
\end{theorem}
\begin{proof}
Since $V_{g} f$ is $ \varepsilon_{\Sigma}$-time-frequency concentrated on $\Sigma$, using  Plancherel's formula (\ref{eq17}), we obtain
$$ \begin{aligned}
\Vert f \Vert_{\ell^2(\Z^n)}^2 \; \Vert g \Vert_{\ell^2(\Z^n)}^2
=\left\Vert V_{g} f \right\Vert_{L^2(\Z^n \times \T^n)}^2 
&= \left \| \chi_{\Sigma^{c}}  V_{g} f \right\|_{L^2(\Z^n \times \T^n)}^2 +	 \left \|  \chi_{\Sigma} V_{g} f \right\|_{L^2(\Z^n \times \T^n)}^2 \\
& \leq\varepsilon_{\Sigma}^2 \left\| V_{g} f\right\|_{L^{2}\left(\Z^n \times \T^n\right)}^2+ \left \|  \chi_{\Sigma} V_{g} f \right\|_{L^2(\Z^n \times \T^n)}^2.
\end{aligned} $$
Hence, using  the relation (\ref{eq21}), we get 
\begin{align}\label{eq30}
(1-\varepsilon_\Sigma^2) \Vert f \Vert_{\ell^2(\Z^n)}^2 \; \Vert g \Vert_{\ell^2(\Z^n)}^2 
&\leq \left \| \chi_{\Sigma} V_{g} f \right\|_{L^2(\Z^n \times \T^n)}^2 \\ \nonumber
& \leq \left\|V_{g} f\right\|_{L^{\infty}\left(\Z^n \times \T^n\right)}^{2} \;(\nu \otimes \mu)(\Sigma)\\ \nonumber
&\leq \|f\|_{\ell^2(\Z^n)}^{2}\; \|g\|_{\ell^2(\Z^n)}^{2}\;(\nu \otimes \mu) (\Sigma),
\end{align}
which completes the proof.
\end{proof}

\begin{theorem}
Let $f \in \ell^2(\Z^n)$ such that $f \neq 0$, $g \in \ell^2(\Z^n)$ be a non-zero window function, $\Sigma \subset \Z^n \times \T^n$ such that  $(\nu \otimes \mu) (\Sigma) < \infty $, and $\varepsilon_\Sigma \geq 0$. If $V_{g} f$ is $ \varepsilon_{\Sigma}$-time-frequency concentrated on $\Sigma$, then for every $p > 2$, we have 
$$(\nu \otimes \mu) (\Sigma) \geq (1-\varepsilon_\Sigma^2)^{\frac{p}{p-2}}.$$
\end{theorem}
\begin{proof}
Since $V_{g} f$ is $ \varepsilon_{\Sigma}$-time-frequency concentrated on $\Sigma$, from \eqref{eq30}, we have 
\begin{align*} 
(1-\varepsilon_\Sigma^2) \Vert f \Vert_{\ell^2(\Z^n)}^2 \; \Vert g \Vert_{\ell^2(\Z^n)}^2 
&\leq \left \|  \chi_{\Sigma} V_{g} f \right\|_{L^2(\Z^n \times \T^n)}^2.
\end{align*}
Again, aplying H\"older’s inequality for the conjugate exponents $\frac{p}{2}$ and $\frac{p}{p-2}$, we get
\begin{align*}
\left \|  \chi_{\Sigma} V_{g} f \right\|_{L^2(\Z^n \times \T^n)}^2&\leq  \left \|  V_{g} f \right\|_{L^p(\Z^n \times \T^n)}^2 ((\nu \otimes \mu) (\Sigma))^{\frac{p-2}{p}}.
\end{align*}
Now, using  the relation \eqref{eq22}, we obtain 
\begin{align*} 
\left \|  \chi_{\Sigma} V_{g} f \right\|_{L^2(\Z^n \times \T^n)}^2&\leq\; \|f\|_{\ell^2(\Z^n)}^{2} \|g\|_{\ell^2(\Z^n)}^{2}((\nu \otimes \mu) (\Sigma))^{\frac{p-2}{p}}.
\end{align*}
Hence, 	$$(\nu \otimes \mu) (\Sigma)\geq (1-\varepsilon_\Sigma^2)^{\frac{p}{p-2}}.$$
\end{proof}

\begin{theorem}
Let $f \in \ell^1(\Z^n)$ such that $\|  V_{g} f \|_{L^2(\Z^n \times \T^n)}=1$, $g \in \ell^2(\Z^n)$ be a non-zero window function, and $\Sigma \subset \Z^n \times \T^n$ such that  $(\nu \otimes \mu) (\Sigma) < \infty $. Let $E \subset \Z^n$ such that $\nu(E) < \infty $. If $f$ is $ \varepsilon_{E}$-concentrated on $E$ in $ \ell^1(\Z^n)$-norm and $V_{g} f$ is $ \varepsilon_{\Sigma}$-time-frequency concentrated on $\Sigma$, then  
$$\nu(E) \geq (1-\varepsilon_E)^2\;\|f\|_{\ell^1(\Z^n )}^2\;\|g\|_{\ell^2(\Z^n)}^2,$$
and 
$$ \|f\|_{\ell^2(\Z^n)}^{2} \;\|g\|_{\ell^2(\Z^n)}^{2}\;(\nu \otimes \mu) (\Sigma) \geq (1-\varepsilon_\Sigma^2).$$
In particular,
$$\nu(E)~ (\nu \otimes \mu) (\Sigma) \geq (1-\varepsilon_E)^2\;(1-\varepsilon_\Sigma^2).$$
\end{theorem}
\begin{proof}
Since $V_{g} f$ is $ \varepsilon_{\Sigma}$-time-frequency concentrated on $\Sigma$, from \eqref{eq30}, we have 
\begin{align*}
(1-\varepsilon_\Sigma^2) \Vert f \Vert_{\ell^2(\Z^n)}^2 \; \Vert g \Vert_{\ell^2(\Z^n)}^2 
&\leq \left \|  \chi_{\Sigma} V_{g} f \right\|_{L^2(\Z^n \times \T^n)}^2.
\end{align*}
Since $\| V_{g} f \|_{L^2(\Z^n \times \T^n)}=1,$ using the relation (\ref{eq21}) and Plancherel's formula (\ref{eq17}), we obtain
\begin{align}\label{eq33}\nonumber
(1-\varepsilon_\Sigma^2)
\leq \left \| \chi_{\Sigma} V_{g} f \right\|_{L^2(\Z^n \times \T^n)}^2
& \leq \left \| V_{g} f \right\|_{L^\infty(\Z^n \times \T^n)}^2 ~ (\nu \otimes \mu) (\Sigma)\\
& \leq \|f\|_{\ell^2(\Z^n)}^{2}~\|g\|_{\ell^2(\Z^n)}^{2}\; (\nu \otimes \mu) (\Sigma).
\end{align}	
Similarly, since $f$ is $ \varepsilon_{E}$-concentrated on $E$ in $ \ell^1(\Z^n)$-norm, using the  Cauchy--Schwarz inequality and the fact that 
$\|f\|_{\ell^2(\Z^n)}\|g\|_{\ell^2(\Z^n)}=1$, we get 
\begin{align}\label{eq34} 
(1-\varepsilon_E) \Vert f \Vert_{\ell^1(\Z^n)} \leq \left \| \chi_{E}f \right\|_{\ell^1(\Z^n )} \leq \left \| f \right\|_{\ell^2(\Z^n )}  \nu (E)^{\frac{1}{2}} 
=\frac{\nu (E)^{\frac{1}{2}}}{\left \| g\right\|_{\ell^2(\Z^n )} } .
\end{align}
Finally, from (\ref{eq33}) and (\ref{eq34}), we have 
$$ \nu(E)~ (\nu \otimes \mu) (\Sigma)\;  \|f\|_{\ell^1(\Z^n)}^{2} \geq  \nu (E)~ (\nu \otimes \mu) (\Sigma)\;  \|f\|_{\ell^2(\Z^n)}^{2}\geq  (1-\varepsilon_E)^2\;(1-\varepsilon_\Sigma^2) \;\|f\|_{\ell^1(\Z^n )}^{2}.$$
Hence,
\[ \nu(E)~ (\nu \otimes \mu) (\Sigma) \geq (1-\varepsilon_E)^2\;(1-\varepsilon_\Sigma^2). \]
This completes the proof of the theorem.
\end{proof}

Let $$|(m, w)|=\sqrt{|m|^2+|w|^2}=\sqrt{|m_1|^2+\cdots+|m_n|^2+|w_1|^2+\cdots+|w_n|^2}.$$ 
For $r>0$, let $B_{r}=\{(m, w) \in \Z^n \times \T^n : ~|(m, w)| \leq r \}$ be the ball with centre at origin  and radius $r$ in $\Z^n \times \T^n$. We have the following estimate for the measure $(\nu \otimes \mu)(B_{r})$ of $B_{r}$ on $\Z^n \times \T^n$:
\begin{align*}
(\nu \otimes \mu)(B_{r}) 
&= \sum_{|m| \leq r} \int_{|w|^2\leq r^2-|m|^2} dw 
= 2 \sum_{|m| \leq r} \int_{0}^{\sqrt{r^2-|m|^2}}dw 
= 2 \sum_{|m| \leq r} \sqrt{r^2-|m|^2} \\
& \leq 2 r \; \text{Card}(\{m \in \Z^n: |m| \leq r \}).
\end{align*}
Notice that the cardinality of the set $\{m \in \Z^n: |m| \leq r \}$ is the number of integer lattice points within a sphere centered at the origin and with the radius $r$, which is famous Gauss's circle problem (see \cite{cil93, hard59}). This number is approximated by the area of the sphere, so the problem is to accurately bound the error term describing how the number of points differs from the area. Some partial results on the upper bound of the error term are available till now. Using the upper bound of the error term, one can obtain an estimate for the cardinality of the set $\{m \in \Z^n: |m| \leq r \}$.

For $s>0$, we define the time-frequency dispersion on $\Z^n \times \T^n$ by
\[ \rho_s(V_gf)=\left(\sum_{m\in\Z^n}\int_{\T^n}|(m,w)|^s \;|V_gf(m,w)|^2 \; dw \right)^{1/s}. \]
In the following proposition, we show that the STFT of an orthonormal sequence cannot be $\varepsilon_{B_{r}}$-time-frequency concentrated on $B_{r}$ with the measure $(\nu \otimes \mu)(B_{r})< \infty$.

\begin{proposition}\label{finite}
Let $g$ be a window function with $ \| g \|_{\ell^2(\Z^n)}=1$ and $B_{r}$ be the ball in $\Z^n \times \T^n$. Let $\varepsilon_{B_{r}}$ and $r$ be two positive real numbers such that $\varepsilon_{B_{r}} < 1$. Let $K$ be a non-empty subset of $\mathbb{N}$ and $\{\phi_{k} \}_{k \in K}$ be an orthonormal sequence in $\ell^2(\Z^n)$. If $V_{g} \phi_{k}$ is $\varepsilon_{B_{r}}$-time-frequency concentrated on $B_{r}$ for every $k \in K$, then $K$ is finite and 
\[ \mathrm{Card} (K) \leq \frac{(\nu \otimes \mu)(B_{r})}{(1-\varepsilon_{B_{r}})}. \]
\end{proposition}
\begin{proof}
Let $M \subset K$ be a non-empty finite subset, then from Theorem \ref{thm3}, we obtain that 
\begin{eqnarray}\label{6} 
\sum_{k \in M} \left(1-\left\|\chi_{B_{r}^{c}} V_g \phi_{k}\right\|_{L^2(\Z^n \times \T^n)}\right) \leq (\nu \otimes \mu)(B_{r}). 
\end{eqnarray}
Also, for every $k \in K$, $V_{g} \phi_{k}$ is $\varepsilon_{B_{r}}$-time-frequency concentrated on $B_{r}$. Therefore
\begin{eqnarray}\label{7}
\left\|\chi_{B_{r}^{c}} V_{g} \phi_{k} \right\|_{L^{2}\left(\Z^n \times \T^n \right)} \leq \varepsilon_{B_{r}} \left\|  V_{g} \phi_{k} \right\|_{L^{2}\left(\Z^n \times \T^n \right)}=\varepsilon_{B_{r}}.
\end{eqnarray}
Hence, by combining relations \eqref{6} and \eqref{7}, we get
\[ \text{Card}(M) \leq \frac{(\nu \otimes \mu)(B_{r})}{(1-\varepsilon_{B_{r}})} . \]
Since $\dfrac{(\nu \otimes \mu)(B_{r})}{(1-\varepsilon_{B_{r}})}$ is independent of $M$, therefore $\text{Card}(K)$ is finite and 
$ \text{Card}(K)\leq \dfrac{(\nu \otimes \mu)(B_{r})}{(1-\varepsilon_{B_{r}})}. $
\end{proof}

\begin{corollary}\label{uniformly}
Let $s>0$ and $g$ be a window function such that $\|g\|_{\ell^2(\Z^n)}=1$. Let $K$ be a non-empty subset of $\mathbb{N}$ and $\{\phi_{k} \}_{k \in K}$ be an orthonormal sequence in $\ell^2(\Z^n)$. If the sequence $\{ \rho_s ( V_{g} \phi_{k} ) \}_{k \in K}$ is uniformly bounded by $A$, then $K$ is finite and $\mathrm{Card}(K) \leq 2 \; (\nu \otimes \mu)(B_{A2^{\frac{2}{s}}})$.  
\end{corollary}
\begin{proof}
Assume that $ \rho_s ( V_{g} \phi_{k} ) \leq A$, for every $k \in K$. Let $r = 4^{\frac{1}{s}} A$. Then
\begin{align*}
\left\|\chi_{B_{r}^{c}} V_{g} \phi_{k} \right\|^2_{L^{2}\left(\Z^n \times \T^n \right)} 
& = \left\| |(m,w)|^{-\frac{s}{2}} \; |(m,w)|^{\frac{s}{2}} \; \chi_{B_{r}^{c}} V_{g} \phi_{k} \right\|^2_{L^{2}\left(\Z^n \times \T^n \right)} \\
& \leq \frac{1}{4 A^s} \left\| |(m,w)|^{\frac{s}{2}} \; \chi_{B_{r}^{c}} V_{g} \phi_{k} \right\|^2_{L^{2}\left(\Z^n \times \T^n \right)} \\
 & \leq \frac{1}{4 A^s} \times A^s 
= \frac{1}{4}\| \phi_{k}\|^2_{\ell^2(\Z^n)} =\frac{1}{4}\|V_{g} \phi_{k} \|^2_{L^2(\Z^n\times\T^n)} .
\end{align*}
Thus, we have
\[ \left\|\chi_{B_{r}^{c}} V_{g} \phi_{k} \right\|_{L^{2}\left(\Z^n \times \T^n \right)} \leq \frac{1}{2}  \left\|  V_{g} \phi_{k} \right\|_{L^{2}\left(\Z^n \times \T^n \right)}. \]
Hence, for every $k \in K$, $V_{g} \phi_{k}$ is $\frac{1}{2}$-time-frequency concentrated on $B_{A2^{\frac{2}{s}}}$. Now, using Proposition \ref{finite}, we obtain that $\text{Card}(K)$ is finite and $\text{Card}(K) \leq 2 \;  (\nu \otimes \mu)(B_{A2^{\frac{2}{s}}})$. 
\end{proof}

\subsection{Benedicks-type uncertainty principle for the STFT}
Here, we study Benedicks-type uncertainty principle for the STFT on $\Z^n \times \T^n$.
First, we show that this transform cannot be concentrated inside a set of measure arbitrarily small.   

\begin{proposition}\label{eq41}
Let $g \in \ell^2(\Z^n)$ be a non-zero window function and $\Sigma\subset \Z^n \times \T^n$. If $\left\|P_{\Sigma} P_{g}\right\|<1$,  then there exists a constant $c(\Sigma, g)>0$ such that for every   $f \in  \ell^2(\Z^n),$ we have
$$\|f\|_{\ell^2(\Z^n)} \|g\|_{\ell^2(\Z^n)}\leq  c(\Sigma, g)\left\|\chi_{\Sigma^{c}} V_{g} f\right\|_{L^{2}\left(\Z^n \times \T^n\right)}. $$
\end{proposition}
\begin{proof}
Since $P_{\Sigma}$ is an orthogonal projection on $L^2(\Z^n \times \T^n)$, for any $F\in L^2(\Z^n \times \T^n)$, we get  
$$
\|P_g(F)\|_{L^{2}\left(\Z^n \times \T^n\right)}^2=\|P_{\Sigma}P_g(F)\|_{L^{2}\left(\Z^n \times \T^n\right)}^2+\|P_{\Sigma^c}P_g(F)\|_{L^{2}\left(\Z^n \times \T^n\right)}^2.
$$
Again, using the identity $P_{\Sigma}P_g(F)=P_{\Sigma}P_g \cdot P_g(F)$, we get $$
\|P_{\Sigma}P_g(F)\|_{L^{2}\left(\Z^n \times \T^n\right)}^2 \leq\|P_{\Sigma}P_g\|^2 \| P_g(F)\|_{L^{2}\left(\Z^n \times \T^n\right)}^2.
$$
Hence, 
\begin{align}\label{eq44}
\|P_g(F)\|_{L^{2}\left(\Z^n \times \T^n\right)}^2\leq \frac{1}{1-\|P_{\Sigma}P_g\|^2 } \|P_{\Sigma^c}P_g(F)\|_{L^{2}\left(\Z^n \times \T^n\right)}^2.
\end{align} 
Since $P_{g}$ is an orthogonal projection from $L^2(\Z^n \times \T^n)$ onto $V_g \left(\ell^2 \left(\Z^n\right)\right)$, for any $f\in \ell^2(\Z^n)$, using the relation  (\ref{eq44}) and Plancherel's formula (\ref{eq17}),  we get 
$$\|f\|_{\ell^2(\Z^n)} \|g\|_{\ell^2(\Z^n)}\leq  \frac{1}{ \sqrt{1-\|P_{\Sigma}P_g\|^2}}\left\|\chi_{\Sigma^{c}} V_{g} f\right\|_{L^{2}\left(\Z^n \times \T^n\right)}.$$
The desired result follows  by choosing the constant $c(\Sigma, g)=  \frac{1}{ \sqrt{1-\|P_{\Sigma}P_g\|^2}}$.
\end{proof}

In the following, we obtain Benedicks-type uncertainty principle for the STFT on $\Z^n \times \T^n$.

\begin{theorem}\label{eq40}
Let $g \in \ell^2(\Z^n)$ be a non-zero window function  such that 
$$\mu \left( \left\{  \mathcal{F}_{\Z^n} (g) \neq 0 \right\} \right)<\infty.$$
Then for any  $\Sigma \subset \Z^n \times \T^n$  such that for almost every $w \in  \T^n, \; \sum\limits_{m \in \Z^n} \chi_{\Sigma}(m, w) <\infty $, we have
$$ V_g \left(\ell^2(\Z^n)\right)  \cap \operatorname{Im} P_{\Sigma}=\{0\}. $$
\end{theorem}
\begin{proof}
Let $F \in V_g \left(\ell^2(\Z^n)\right) \cap \operatorname{Im} P_{\Sigma}$ be a non-trivial function. Then there exists a function $f \in \ell^2(\Z^n)$ such that $F=	V_{g} f$ and $\supp F \subset \Sigma$. For any $w \in \T^n$, we consider the function
$$ F_{w}(m)=V_{g} f(m, w),\quad m \in \Z^n. $$ 
Then, we get $\supp F_{w} \subset\{m \in \Z^n :  (m, w ) \in \Sigma \}.$ Since for almost every $w \in \T^n,  \sum\limits_{m \in \Z^n} \chi_{\Sigma}(m, w) < \infty$, we have $ \nu \left( \supp F_{w}\right)<\infty$. Since $V_g f(m,w)=(M_{-w} f * \overline{\tilde{g}})(m)$, we get
$$\mathcal{F}_{\Z^n} (F_{w})= \mathcal{F}_{\Z^n} (M_{-w} f) \;\; \mathcal{F}_{\Z^n} (\overline{\tilde{g}}) \quad \;a.e. $$
Hence,
$$ \supp \mathcal{F}_{\Z^n} (F_{w}) \subset  \supp   \mathcal{F}_{\Z^n} (\overline{\tilde{g}} ),$$ 
and from the hypothesis, we get  $\mu \left(\left\{\mathcal{F}_{\Z^n} (F_{w}) \neq 0\right\}\right)<\infty$. From Benedicks--Amrein--Berthier's uncertainty principle \cite{hog88}, we conclude that, for every  $w \in \T^n, \; F(\cdot, w)=0$, which eventually  implies that $F=0.$
\end{proof}

\begin{remark}
Let $g \in \ell^2(\Z^n)$ be a non-zero window function  such that  $$\mu \left(  \left\{ \mathcal{F}_{\Z^n} (g)  \neq 0 \right\} \right)<\infty.$$ Then, for any non-zero function $f \in \ell^2(\Z^n)$, we have $(\nu \otimes \mu) ( \supp V_{g} f)=\infty$, i.e., the support of $V_{g} f$ cannot be of   finite measure. If $(\nu \otimes \mu) ( \supp V_{g} f) < \infty$, then from Theorem \ref{eq40}, we get $f \equiv 0$ or $g \equiv 0$.
\end{remark}

\begin{proposition}
Let $g \in \ell^2(\Z^n)$ be a non-zero window function  such that  $$\mu \left(  \left\{\mathcal{F}_{\Z^n} (g)  \neq 0 \right\} \right)<\infty.$$  Let $\Sigma \subset \Z^n \times \T^n$ such that $(\nu \otimes \mu) (\Sigma)<\infty $,  then there exists a constant $c(\Sigma, g)>0$ such that
$$
\|f\|_{\ell^2(\Z^n)} \|g\|_{\ell^2(\Z^n)} \leq c(\Sigma, g ) 	\left\|\chi_{\Sigma^{c}} V_{g} f\right\|_{L^{2}\left(\Z^n \times \T^n\right)}.
$$
\end{proposition}
\begin{proof}
Assume that $$\left\|P_{\Sigma} P_{g}(F)\right\|_{L^{2}\left(\Z^n \times \T^n\right)} =\|F\|_{L^{2}\left(\Z^n \times \T^n\right)}, \;F \in L^2(\Z^n \times \T^n).$$  Since $P_{\Sigma}$ and $P_{g}$ are  orthogonal projections, we obtain $P_{\Sigma}(F)=P_{g}(F)=F.$ Again, since $(\nu \otimes \mu) (\Sigma)<\infty $, for almost every $w \in  \T^n, \;\;  \sum\limits_{m \in \Z^n} \chi_{\Sigma}(m, w) <\infty$ and from Theorem \ref{eq40}, we get $F=0.$ Hence,  for $F \neq 0$, we obtain  $$\left\|P_{\Sigma} P_{g}(F)\right\|_{L^{2}\left(\Z^n \times \T^n\right)} <\|F\|_{L^{2}\left(\Z^n \times \T^n\right)}. $$  
Since  $P_{\Sigma} P_{g}$ is a Hilbert--Schmidt operator, we obtain that the   largest eigenvalue $|\lambda|$ of the operator $P_{\Sigma} P_{g}$ satisfy $|\lambda|<1$ and $\left\|P_{\Sigma} P_{g}\right\|=|\lambda|<1$. Finally, using Proposition \ref{eq41},  we get the desired result. 
\end{proof}

\subsection{Heisenberg-type uncertainty principle for the STFT} 
In this subsection, we study Heisenberg-type uncertainty inequality for the STFT on $\Z^n \times \T^n$ for general magnitude $s > 0.$  Indeed, we have the following result.
\begin{theorem}\label{eq39}
Let $s>0$. Then there exists a constant $c(s)>0$ such that for all $f, g \in \ell^2(\Z^n)$, we have
\begin{align*}
\left\||m|^s V_{g} f\right\|_{L^2(\Z^n \times \T^n)}^2 +\left\||w|^s V_{g} f\right\|_{L^2(\Z^n \times \T^n)}^2 
\geq  c(s)\; \|f\|_{\ell^2(\Z^n)}^{2}\;\|g\|_{\ell^2(\Z^n)}^{2}.
\end{align*}
\end{theorem}
\begin{proof}
Let $r>0$  and $B_{r}=\{(m, w) \in \Z^n \times \T^n : ~|(m, w)|<r \}$ be the ball with centre at origin and radius $r$ in $\Z^n \times \T^n$. Fix $\varepsilon_{0} \leq 1$ small enough such that $(\nu \otimes \mu) (B_{\varepsilon_{0}})<1$. From Proposition \ref{eq35}, we have 
\begin{align*}
& \|f\|_{\ell^2(\Z^n)}^{2}\; \|g\|_{\ell^2(\Z^n)}^{2}\\
& \leq \frac{1}{ \left[1-(\nu \otimes \mu)(B_{\varepsilon_{0}})\right]} {\sum \int}_{\{(m, w) \in \Z^n \times \T^n : ~|(m, w)| \geq  \varepsilon_{0}\}} \left|V_{g} f(m, w)\right|^{2} ~d\mu  (w) \\
& \leq \frac{1}{\varepsilon_{0}^{2 s}\left[1-(\nu \otimes \mu)(B_{\varepsilon_{0}})\right]} {\sum \int}_{\{(m, w) \in \Z^n \times \T^n : ~|(m, w)| \geq  \varepsilon_{0}\}}|(m, w)|^{2 s}\left|V_{g} f(m, w)\right|^{2} ~d\mu(w) \\
& \leq \frac{1}{\varepsilon_{0}^{2 s}\left[1-(\nu \otimes \mu)(B_{\varepsilon_{0}})\right]} \sum_{m \in \Z^n}  \int_{\T^n} |(m, w)|^{2 s}\left|V_{g} f(m, w)\right|^{2} ~d\mu(w).
\end{align*}
Consequently,
\begin{align}\label{eq36}
\left\||(m, w)|^{s}V_{g} f\right\|_{L^2(\Z^n \times \T^n)}^2  \geq \varepsilon_{0}^{2 s}\left[1-(\nu \otimes \mu)(B_{\varepsilon_{0}})\right]\; \|f\|_{\ell^2(\Z^n)}^{2}\;\|g\|_{\ell^2(\Z^n)}^{2}.
\end{align}
Finally, using  the fact that $|a+b|^{s} \leq 2^{s}\left(|a|^{s}+|b|^{s}\right)$ and  from (\ref{eq36}), we get 
\begin{align*}
\varepsilon_{0}^{2 s}\left[1-(\nu \otimes \mu)(B_{\varepsilon_{0}})\right]\; \|f\|_{\ell^2(\Z^n)}^{2}\;\|g\|_{\ell^2(\Z^n)}^{2}
& \leq \left\||(m, w)|^{s}V_{g} f\right\|_{L^2(\Z^n \times \T^n)}^2\\
& \leq 2^s \left\| |m|^s V_{g} f\right\|_{L^2(\Z^n \times \T^n)}^2  +2^s\left\| |w|^s V_{g} f\right\|_{L^2(\Z^n \times \T^n)}^2 .
\end{align*}
Hence, we get the desired result with  $c(s)=\varepsilon_{0}^{2 s}2^{-s}\left[1-(\nu \otimes \mu)(B_{\varepsilon_{0}})\right].$ 
\end{proof}

\subsection{Local uncertainty inequality for the STFT}
Here, we discuss results related to  the $L^2(\Z^n \times \T^n)$-mass of the STFT outside sets of finite measure. Indeed, we establish the following result.
\begin{theorem} \label{eq37}
Let $s>0 $. Then there exists a constant $c(s)>0$ such that for any $f, g \in \ell^2(\Z^n)$ and any $\Sigma  \subset \Z^n \times \T^n$  such that $(\nu \otimes \mu)(\Sigma)<\infty$, we have
$$
\left\|V_{g} f\right\|_{L^{2}\left(\Sigma \right)} \leq c(s)\left[(\nu \otimes \mu)(\Sigma)\right]^{\frac{1}{2}}\left\||(m, w)|^{s}V_{g} f\right\|_{L^2(\Z^n \times \T^n)}.
$$
\end{theorem}
\begin{proof}
Since $ \|V_{g} f \|_{L^{2}\left(\Sigma \right)}\leq  \|V_{g} f \|_{L^{\infty}\left(\Z^n \times \T^n\right)}~ [(\nu \otimes \mu)(\Sigma)]^{\frac{1}{2}},$
using the relation (\ref{eq21}),  we get 
$$\left\|V_{g} f\right\|_{L^{2}\left(\Sigma \right)} \leq \left[(\nu \otimes \mu)(\Sigma)\right]^{\frac{1}{2}}\; \|f\|_{\ell^2(\Z^n)}~\|g\|_{\ell^2(\Z^n)}.$$
Thus, using inequality (\ref{eq36}), we get 
\begin{align*}
\left\|V_{g} f\right\|_{L^{2}\left(\Sigma \right)} \leq \frac{\left[(\nu \otimes \mu)(\Sigma)\right]^{\frac{1}{2}}}{\varepsilon_{0}^{s}\left[1-(\nu \otimes \mu)(B_{\varepsilon_{0}})\right]^{\frac{1}{2}}}\left\||(m, w)|^{s}V_{g} f\right\|_{L^2(\Z^n \times \T^n)}.
\end{align*}
This completes the proof of the theorem with $c(s)=\varepsilon_{0}^{-s}\left[1-(\nu \otimes \mu)(B_{\varepsilon_{0}})\right]^{-\frac{1}{2}}.$
\end{proof}

In the following, we show that Theorem \ref{eq37} gives a general form of Heisenberg-type inequality with a different constant.

\begin{corollary}
Let $s>0 .$ Then there exists a constant $c_{s}>0$ such that, for all $f, g \in$ $\ell^2(\Z^n)$, we have 
\begin{align}\label{eq38}
\left\||(m, w)|^{s}V_{g} f\right\|_{L^2(\Z^n \times \T^n)}\geq  c_{s} ~\|f\|_{\ell^2(\Z^n)}\;\|g\|_{\ell^2(\Z^n)}.
\end{align}
In particular,
\begin{align*}
\left\||m|^s V_{g} f\right\|_{L^2(\Z^n \times \T^n)}^2  +\left\||w|^s V_{g} f\right\|_{L^2(\Z^n \times \T^n)}^2 
\geq  \frac{c_{s}^2}{2^s}\; \|f\|_{\ell^2(\Z^n)}^{2}\;\|g\|_{\ell^2(\Z^n)}^{2}.\end{align*}
\end{corollary} 
\begin{proof}
Let $r>0$  and $B_{r}=\{(m, w) \in \Z^n \times \T^n : ~|(m, w)|<r \}$ be the ball with  centre at origin  and radius $r$ in $\Z^n \times \T^n$. Using Plancherel's formula (\ref{eq17}) and Theorem \ref{eq37}, we obtain
$$
\begin{aligned}
& \Vert f \Vert_{\ell^2(\Z^n)}^2 \; \Vert g \Vert_{\ell^2(\Z^n)}^2
=\left\Vert V_{g} f \right\Vert_{L^2(\Z^n \times \T^n)}^2 
=\left \|  \chi_{B_r} V_{g} f \right\|_{L^2(\Z^n \times \T^n)}^2+\left \| \chi_{B_r^{c}}  V_{g} f \right\|_{L^2(\Z^n \times \T^n)}^2  \\
& \leq c(s)^2\left[(\nu \otimes \mu)(B_r)\right]\left\||(m, w)|^{s}V_{g} f\right\|_{L^2(\Z^n \times \T^n)}^2  + r^{-2s}\left\||(m, w)|^{s}V_{g} f\right\|_{L^2(\Z^n \times \T^n)}^2.
\end{aligned}
$$
We get  the inequality  (\ref{eq38}) by minimizing the right-hand side of the above  inequality over $r>0$. Now, proceeding  similarly  as in the proof of Theorem 
\ref{eq39}, we obtain	
\begin{align*}
 2^s\left\| |m|^s V_{g} f\right\|_{L^2(\Z^n \times \T^n)}^2  +2^s\left\| |w|^s V_{g} f\right\|_{L^2(\Z^n \times \T^n)}^2 
\geq  c_{s}^2\; \|f\|_{\ell^2(\Z^n)}^{2}~\|g\|_{\ell^2(\Z^n)}^{2},
\end{align*} 
and this completes the proof.
\end{proof}

\subsection{Heisenberg-type uncertainty inequality via the $k$-entropy}
In this subsection, we study the localization of the $k$-entropy of the STFT over the space $\Z^n \times \T^n$. We first need the following definition.
\begin{definition}\quad
\begin{enumerate} 
\item A probability density function $\rho$ on $\Z^n \times \T^n$ is a non-negative measurable function on $\Z^n \times \T^n$ satisfying
$$ \sum_{m \in \Z^n}  \int_{\T^n} \rho(m, w)~d\mu(w) =1. $$
\item Let $\rho$  be a probability density function  on $\Z^n \times \T^n$. Then the $k$-entropy of  $\rho$   is defined by
$$ E_{k}(\rho):=- \sum_{m \in \Z^n}  \int_{\T^n} \ln (\rho(m, w))~ \rho(m, w)~d\mu(w),$$
whenever the integral on the right-hand side is well defined. 
\end{enumerate}
\end{definition}
Now, we prove the main result of this subsection.
\begin{theorem}
Let $g \in \ell^2(\Z^n)$ be a non-zero window function and $f \in \ell^2(\Z^n)$ such that $f\neq 0$. Then, we have 
\begin{align}\label{eq42}
E_{k}(|V_{g} f|^{2}) \geq -2 \ln \left(\|f\|_{\ell^2(\Z^n)} \;\|g\|_{\ell^2(\Z^n)}\right)  \|f\|_{\ell^2(\Z^n)}^{2}\;\|g\|_{\ell^2(\Z^n)}^{2}.
\end{align}
\end{theorem}
\begin{proof}
First, we assume that $ \|f\|_{\ell^2(\Z^n)}=\|g\|_{\ell^2(\Z^n)}=1$. For any $(m, w) \in \Z^n \times \T^n$, using  the relation (\ref{eq21}), we get 
$$ |V_{g} f(m, w)| \leq \; \|f\|_{\ell^2(\Z^n)}\;\|g\|_{\ell^2(\Z^n)}=1. $$
Consequently, $\ln (|V_{g} f|)\leq 0$ and   therefore  $E_{k} \big(|V_{g} f|\big) \geq 0$. The desired inequality (\ref{eq42}) holds  trivially  if the entropy $E_{k} \big(|V_{g} f|\big)$ is infinite. Now, suppose  that the entropy $E_{k} \big(|V_{g} f|\big)$ is finite. Let $ f$ and $g  $ be two non-zero functions in $\ell^2(\Z^n)$. We define 
$$
\phi=\frac{f}{\|f\|_{\ell^2(\Z^n)}} ~\text { and }~ \psi=\frac{g}{\|g\|_{\ell^2(\Z^n)}}.
$$
Then $\phi , ~ \psi \in  \ell^2(\Z^n)$ with $\|\phi\|_{\ell^2(\Z^n)}=\|\psi\|_{\ell^2(\Z^n)}=1$ and therefore 
\begin{align}\label{eq43}
E_{k} \big(|V_{\psi} \phi|\big) \geq 0.
\end{align} 
Since  $V_{\psi} \phi =\frac{1}{\|f\|_{\ell^2(\Z^n)}\|g\|_{\ell^2(\Z^n)} } V_{g} f$, we have 
\begin{align*}
E_{k} \big(|V_{\psi} \phi|^{2}\big)
& = - \sum_{m \in \Z^n}  \int_{\T^n} \ln (|V_{\psi} \phi (m, w)|^{2})~ |V_{\psi} \phi (m, w)|^{2} ~d\mu(w) \\
& =\frac{1}{ \|f\|_{\ell^2(\Z^n)}^2\|g\|_{\ell^2(\Z^n)}^2}  E_{k} ( |V_{g} f | ^{2} ) +2 \ln \big(\|f\|_{\ell^2(\Z^n) }\;\|g\|_{\ell^2(\Z^n) }\big).
\end{align*}
Finally,  from (\ref{eq43}),  we obtain  
$$
E_{k} ( |V_{g} f |^{2} ) \geq -2 \ln \left( \|f\|_{\ell^2(\Z^n)}\;\|g\|_{\ell^2(\Z^n)}\right) \|f\|_{\ell^2(\Z^n)}^{2}\;\|g\|_{\ell^2(\Z^n)}^{2}.
$$
This completes the proof of the theorem.
\end{proof}

\section*{Acknowledgments}
The authors wish to thank the anonymous referees for their helpful comments and suggestions that helped to improve the quality of the paper.


\begin{thebibliography}{10}

\bibitem{ben95}  
M. Benedicks, {\it On Fourier transforms of functions supported on sets of finite Lebesgue measure}, J. Math. Anal. Appl. 106:180--183 (1985).

\bibitem{bial85} 
I. Bialynicki-Birula, {\it Entropic Uncertainty Relations in Quantum Mechanics}, Quantum Probability and Applications II,  Springer, Berlin 90--103 (1985)

\bibitem{bona03}  A.  Bonami,  B. Demange and P. Jaming, {\it Hermite functions and uncertainty principles for the Fourier and the windowed Fourier transforms},  Rev. Mat. Iberoam. 19:23--55 (2003).

\bibitem{cil93}
J. Cilleruelo, {\it The distribution of the lattice points on circles}, J. Number Theory 43(2):198--202 (1993).

\bibitem{cow83}  
M.G. Cowling and  J.F. Price, {\it Generalizations of Heisenberg's inequality},  In: Harmonic analysis (G. Mauceri, F. Ricci and G. Weiss  (eds.)),  Springer,  Berlin 443--449 (1983).
 
\bibitem{das22}
A. Dasgupta and A. Poria, {\it Localization operators on discrete modulation spaces}, arXiv:2202.10791 (2022).

\bibitem{dono89}  
D.L. Donoho and  P.B. Stark, {\it Uncertainty principles and signal recovery},  SIAM J. Appl. Math. 49(3):906--931 (1989).

\bibitem{fol97} 
G.B. Folland and  A. Sitaram, {\it The uncertainty principle: A mathematical survey},  J. Fourier Anal. Appl. 3:207--238 (1997). 

\bibitem{gro01} 
K. Gr\"ochenig, {\it Foundations of Time-Frequency Analysis}, Birkh\"auser, Boston (2001).

\bibitem{Gro03}
K. Gr\"ochenig, {\it Uncertainty principles for time--frequency representations}. In: Advances in Gabor analysis (H. G. Feichtinger and T. Strohmer (eds.)), pp. 11--30, Birkh\"auser, Boston (2003).

\bibitem{Gro01}
K. Gr\"ochenig and G. Zimmermann, {\it Hardy's theorem and the short-time Fourier transform of Schwartz functions}, J. Lond. Math. Soc. (2) 63(1):205--214 (2001).

\bibitem{har33} 
G.H. Hardy,  {\it A theorem concerning Fourier transforms}, J. Lond. Math. Soc. 8:227--231 (1933).

\bibitem{hard59}
G.H. Hardy, {\it Ramanujan: twelve lectures on subjects suggested by his life and work}, Chelsea Publishing Company, New York (1959). 

\bibitem{havin}
V. Havin  and B. J\"oricke,  {\it The uncertainty principle in harmonic analysis}, Springer-Verlag, Berlin (1994).

\bibitem{hog88}
J. Hogan, {\it A qualitative uncertainty principle for locally compact Abelian groups}, Proc. Centre Math. Anal. Austral. Nat. Univ. 16:133--142 (1988).

\bibitem{hor1991}  
L.  H\"ormander,  {\it A uniqueness theorem of Beurling for Fourier transform pairs}, Ark. Mat. 29:237--240 (1991).

\bibitem{mor34} 
M.W. Morgan, {\it A note on Fourier transforms}, J. Lond. Math. Soc. 9:188--192 (1934).

\bibitem{ric14} 
B. Ricaud and B. Torrsani, {\it A survey of uncertainty principles and some signal processing applications}, Adv. Comput. Math. 40:629--650 (2014).
 
\bibitem{sle61} 
D. Slepian and H.O. Pollak, {\it Prolate spheroidal wave functions, Fourier analysis and uncertainty I}, Bell Syst. Tech. J. 40:43--63 (1961).

\bibitem{ste56}
E.M. Stein, {\it Interpolation of linear operators}, Trans. Amer. Math. Soc. 83:482--492 (1956).

\bibitem{Wilc00} 
E. Wilczok, {\it New uncertainty principles for the continuous Gabor transform and the continuous Wavelet transform}, Doc. Math. 5:201--226 (2000).

\end{thebibliography}
\end{document}